\documentclass{amsart}
\usepackage{amsthm}
\usepackage[dvipdfmx]{graphicx}

\title[a refinement of Johnson's bounding]{a refinement of Johnson's bounding for the stable genera of Heegaard splittings}
\author{KAZUTO Takao}
\address{Department of Mathematics, Graduate school of science, Osaka University.}
\email{u713544f@ecs.cmc.osaka-u.ac.jp}
\theoremstyle{definition}
\newtheorem{Def}{Definition}
\theoremstyle{plain}
\newtheorem{Thm}[Def]{Theorem}
\newtheorem{Lem}[Def]{Lemma}
\newtheorem{Prop}[Def]{Proposition}

\begin{document}
\begin{abstract}
For each integer $k\geq 2$, Johnson gave a $3$-manifold with Heegaard splittings of genera $2k$ and $2k-1$ such that any common stabilization of these two surfaces has genus at least $3k-1$.
We modify his argument to produce a $3$-manifold with two Heegaard splitings of genus $2k$ such that any common stabilization of them has genus at least $3k$.
\end{abstract}
\maketitle

\section{Introduction}

A {\it genus} $g$ {\it Heegaard splitting} for a closed $3$-manifold $M$ is a triple $(\Sigma ,H^-,H^+)$ where $H^-,H^+$ are genus $g$ handlebodies such that $H^-\cup H^+=M$ and $H^-\cap H^+=\partial H^-=\partial H^+ =\Sigma $.
The genus $g$ surface $\Sigma $ is called the {\it Heegaard surface}.
Any closed, orientable, connected $3$-manifold has Heegaard splittings.
Two Heegaard splittings for the same $3$-manifold are called {\it isotopic} if there is an ambient isotopy taking one of the Heegaard surfaces to the other.

Suppose $\alpha $ is a properly embedded arc in $H^+$ parallel to $\Sigma $.
Add a regular neighborhood of $\alpha $ to $H^-$ and delete it from $H^+$.
Then the result is a new Heegaard splitting whose genus is one greater than that of the original.
A {\it stabilization} of a Heegaard splitting is another splitting obtained by a finite sequence of such processes.
Any two Heegaard splittings of the same $3$-manifold have a common stabilization \cite{reidemeister}, \cite{singer}.
That is to say, there is a third Heegaard splitting which is isotopic to a stabilization of each of the initial splittings.
The {\it stable genus} of two Heegaard splittings is the minimal genus of their common stabilizations.

It had been conjectured that the stable genus of any two Heegaard splittings is at most $p+1$, where $p$ is the larger of the two initial genera, which is called {\it the Stabilization Conjecture}.
This conjecture has been verified for many classes of $3$-manifolds, including Seifert fibered spaces \cite{schultens}, most genus-two $3$-manifolds \cite{rubinstein2} (see also \cite{berge}) and most graph manifolds \cite{derby1} (see also \cite{sedgwick}).

Johnson \cite{johnson2} gave a counterexample for this conjecture.
For each $k\geq 2$, he constructed an irreducible toroidal $3$-manifold with Heegaard splittings of genera $2k-1$ and $2k$ such that the stable genus of these two splittings is $3k-1$.
In fact, we can see that the stable genus is at most $3k-1$ by a simple observation, and the point is the bounding from below.
His construction can be easily modified to produce an atoroidal $3$-manifold with Heegaard splittings of genera $2k-n$ and $2k$ whose stable genus is $3k-n$, where $n$ is larger than $1$.
However, the larger $n$ is, the closer the stable genus is to the genus of the original.
If $n$ is larger than $k-2$, it does not give a counterexample for the conjecture.
We modify his construction to the opposite direction and refine the bounding for the stable genus from bellow as the following:

\begin{Thm}\label{main}
For every $k\geq 2$, there exists a $3$-manifold with two Heegaard splittings of genus $2k$ whose stable genus is $3k$.
\end{Thm}

This $3$-manifold is reducible.
Actually, we get it by taking connected sum of two closed $3$-manifolds with Heegaard splittings of genus $k$ with high Hempel distance (see Section \ref{SplittingS}).
It may be a strong point of this paper that we can construct a counterexample for the Stabilization Conjecture from genus-two $3$-manifolds by substituting $2$ for $k$.
There are fairly many studies on genus-two $3$-manifolds.
For instance, Kobayashi \cite{kobayashi1} gave a complete list of genus-two $3$-manifolds admitting nontrivial torus decompositions.

Prior to Johnson \cite{johnson2}, a counterexample for the ``oriented version" of the Stabilization Conjecture was given by Hass, Thompson and Thurston \cite{hass}.
In the ``oriented version", two Heegaard splittings are called isotopic only if the isotopy preserves the order of the handlebodies.
For a Heegaard splitting, the minimal genus of its stabilizations where the handlebodies can be interchanged by an isotopy is called the {\it flip genus}.
They showed that there is a Heegaard splitting whose flip genus is twice the initial genus.

For the oriented version, Johnson \cite{johnson1} gave an estimate for general Heegaard splittings.
He showed that the flip genus of any Heegaard splitting of genus $k$ with Hempel distance $d$ is at least ${\rm min}\{ 2k,\frac{1}{2}d\} $.
His counterexample in \cite{johnson2} and ours for the non-oriented version can be viewed as applications of this estimation.

Bachman \cite{bachman2} also gave several counterexamples using different techniques.
One is for the oriented version, and another is for the non-oriented version.

I would like to express my appreciation to Ken'ichi Ohshika, Tsuyoshi Kobayashi and Makoto Sakuma for their advices and encouragement.
I would also like to thank Jesse Johnson for helpful comments.

\section{Heegaard splittings}

To begin with, we will define Heegaard splittings for compact $3$-manifolds possibly with boundaries.
A {\it compression body} is a connected $3$-manifold $H$ which can be obtained from $S\times [0,1]$ by attaching finitely many $1$-handles to $S\times \{ 1\}$ where $S$ is a closed, orientable, possibly disconnected surface.
We will use the notations like $\partial _-H=S\times \{ 0\} $ and $\partial _+H=\partial H\setminus \partial _-H$.
Handlebodies are regarded as the extreme cases of compression bodies, i.e. $\partial _-H=\emptyset $.
A {\it Heegaard splitting} for a compact $3$-manifold $M$ is a triple $(\Sigma ,H^-,H^+)$ where $H^-,H^+$ are compression bodies such that $H^-\cup H^+=M$ and $H^-\cap H^+=\partial _+H^-=\partial _+H^+ =\Sigma $.
The {\it genus} of $(\Sigma ,H^-,H^+)$ is the genus of $\Sigma $.

In addition to stabilizations, we will use some sorts of operations to construct new Heegaard splittings from given Heegaard splittings.
Now, we will define such operations in the next three paragraphs:

Suppose $(\Sigma _1,H^-_1,H^+_1)$ and $(\Sigma _2,H^-_2,H^+_2)$ are Heegaard splittings for compact $3$-manifolds $M_1$ and $M_2$, respectively.
Let $B_i$ be a ball in $M_i$ such that $\Sigma _i\cap B_i$ is an equatorial plane of $B_i$ for each $i=1,2$.
Suppose $\varphi :\partial B_1\rightarrow \partial B_2$ is a homeomorphism such that $\varphi (H^-_1\cap \partial B_1)=H^-_2\cap \partial B_2$ and $\varphi (H^+_1\cap \partial B_1)=H^+_2\cap \partial B_2$.
Let $M$ be the $3$-manifold obtained by gluing the closures of $M_1\setminus B_1$ and $M_2\setminus B_2$ by $\varphi $, namely, the connected sum of $M_1$ and $M_2$.
Let $H^-$ be the compression body obtained by gluing the closures of $H^-_1\setminus B_1$ and $H^-_2\setminus B_2$ by $\varphi $ and let $H^+$ be the compression body obtained by gluing the closures of $H^+_1\setminus B_1$ and $H^+_2\setminus B_2$ by $\varphi $.
Then $(\Sigma ,H^-,H^+)$ is a Heegaard splitting for $M$ where $\Sigma =\partial _+H^-=\partial _+H^+$.
It is called the {\it connected sum} of $(\Sigma _1,H^-_1,H^+_1)$ and $(\Sigma _2,H^-_2,H^+_2)$.

Suppose $M_1,M_2,(\Sigma _1,H^-_1,H^+_1)$ and $(\Sigma _2,H^-_2,H^+_2)$ are as above.
Suppose $\partial _-H^+_1$ is non-empty and homeomorphic to $\partial _-H^+_2$.
Let $M$ be the union of $M_1$ and $M_2$ identifying $\partial _-H^+_1$ with $\partial _-H^+_2$ by some homeomorphism.
Since $H^+_i$ is a compression body, it can be decomposed into a product manifold $\partial _-H^+_i\times [0,1]$ and a collection of $1$-handles for each $i=1,2$.
The part $(\partial _-H^+_1\times [0,1])\cup (\partial _-H^+_2\times [0,1])$ of $M$ can be collapsed without changing the topology of $M$.
Then we can regard the $1$-handles which belonged to $H^+_1$ are attached to $H^-_2$, forming a new compression body $H^+$.
Similarly, $H^-_1$ and the $1$-handles which belonged to $H^+_2$ form another compression body $H^-$.
Then $(\Sigma ,H^-,H^+)$ is a Heegaard splitting for $M$ where $\Sigma =\partial _+H^-=\partial _+H^+$.
We will say that $(\Sigma ,H^-,H^+)$ is the {\it amalgamation} of $(\Sigma _1,H^-_1,H^+_1)$ and $(\Sigma _2,H^-_2,H^+_2)$.
Note that $H^-_1\subset H^-, H^-_2\subset H^+$ and $(\Sigma ,H^+,H^-)$ is the amalgamation of $(\Sigma _2,H^-_2,H^+_2)$ and $(\Sigma _1,H^-_1,H^+_1)$.

Suppose $M$ is a compact $3$-manifold with a single boundary component, and $(\Sigma ,H^-,H^+)$ is a Heegaard splitting for $M$ such that $\partial _-H^+=\partial M$.
Decompose $H^+$ into a product manifold $\partial _-H^+\times [0,1]$ and a collection of $1$-handles.
Let $\alpha $ be a vertical arc in $\partial _-H^+\times [0,1]$.
Add a neighborhood of the union of $\alpha $ and $\partial _-H^+$ to $H^-$, to obtain a compression body $H^{\prime +}$.
Then the closure of the complement of $H^{\prime +}$ in $M$ is homeomorphic to the union of $(\partial _-H^+\setminus (\text{an open disk}))\times [0,1]$ and $1$-handles.
This is a handlebody, denoted by $H^{\prime -}$.
We will call $(\Sigma ',H^{\prime -},H^{\prime +})$ the {\it boundary stabilization} of $(\Sigma ,H^-,H^+)$ where $\Sigma '=\partial H^{\prime -}=\partial _+H^{\prime +}$.
We are afraid the labels of $H^{\prime -}$ and $H^{\prime +}$ are confusing, but we would like to keep the condition that $\partial M$ is contained in the latter compression body.

Johnson's counterexample was constructed by amalgamations along the torus boundaries.
All his arguments in \cite{johnson2} can be applied also if the boundaries have genus more than one.
We will make the same construction changing the place of torus boundaries by sphere boundaries.
Though it is common in theories on Heegaard splittings to assume that the $3$-manifolds do not have sphere boundaries, we do not have to do so at least in the above definitions.
It is useful in our arguments to deal with amalgamations along sphere boundaries while they are no other than connected sums as the following:

\begin{Prop}\label{prop1}
Suppose $(\Sigma _i,H^-_i,H^+_i)$ is a Heegaard splitting for a closed $3$-manifold $M_i$, and $B_i$ is an open ball in $H^+_i$ for $i=1,2$.
Then the amalgamation of $(\Sigma _1,H^-_1,H^+_1\setminus B_1)$ and $(\Sigma _2,H^-_2,H^+_2\setminus B_2)$ is isotopic (in the oriented version) to the connected sum of $(\Sigma _1,H^-_1,H^+_1)$ and $(\Sigma _2,H^+_2,H^-_2)$.
\end{Prop}

\begin{proof}
See pictures in the next page.
In Figure \ref{fig5}, $H^+_1$ is regarded as a ball attached $1$-handles while $H^-_1$ as its complement.
In Figure \ref{fig6}, $H^+_2$ and $H^-_2$ are figured similarly but inside out.
The handlebodies $H^+_1,H^-_2$ are painted gray and $B_1,B_2$ are patterned with meshes.
The amalgamation is constructed by gluing $M_1\setminus B_1$ and $M_2\setminus B_2$ as Figure \ref{fig7} and collapsing the product part as Figure \ref{fig8}.
On the other hand, choose a ball $B'_i$ which intersects $\Sigma _i$ in a disk for each $i=1,2$ as Figures \ref{fig9}, \ref{fig10}.
The connected sum is constructed by gluing $M_1\setminus B'_1$ and $M_2\setminus B'_2$ as Figure \ref{fig11}, which is equivalent to Figure \ref{fig8}.
\end{proof}

\begin{figure}[ht]
\begin{tabular}{ccc}
\begin{minipage}{3.8cm}
\begin{center}
\caption{}
\vspace{.2cm}
\includegraphics[width=2.5cm]{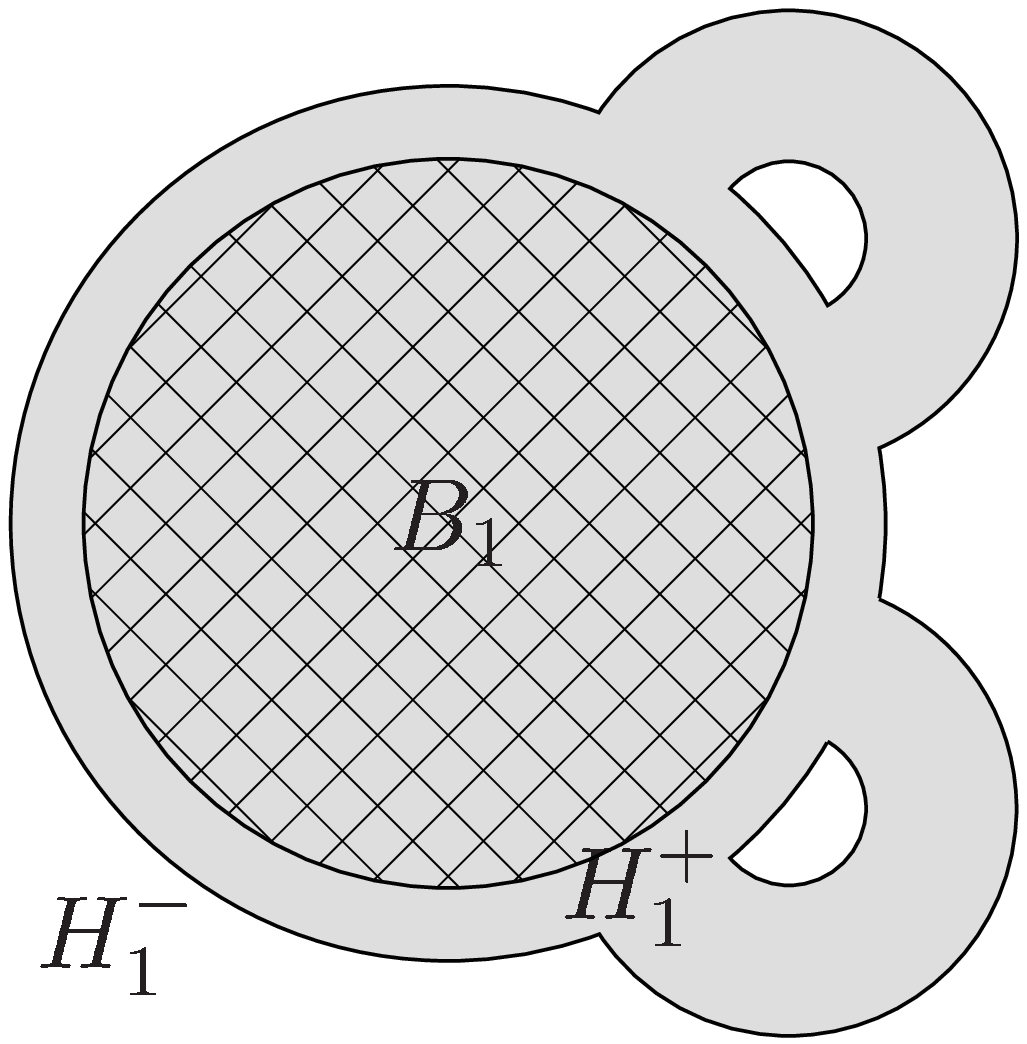}
\label{fig5}
\end{center}
\end{minipage}
\begin{minipage}{3.8cm}
\begin{center}
\caption{}
\vspace{.2cm}
\includegraphics[width=2.5cm]{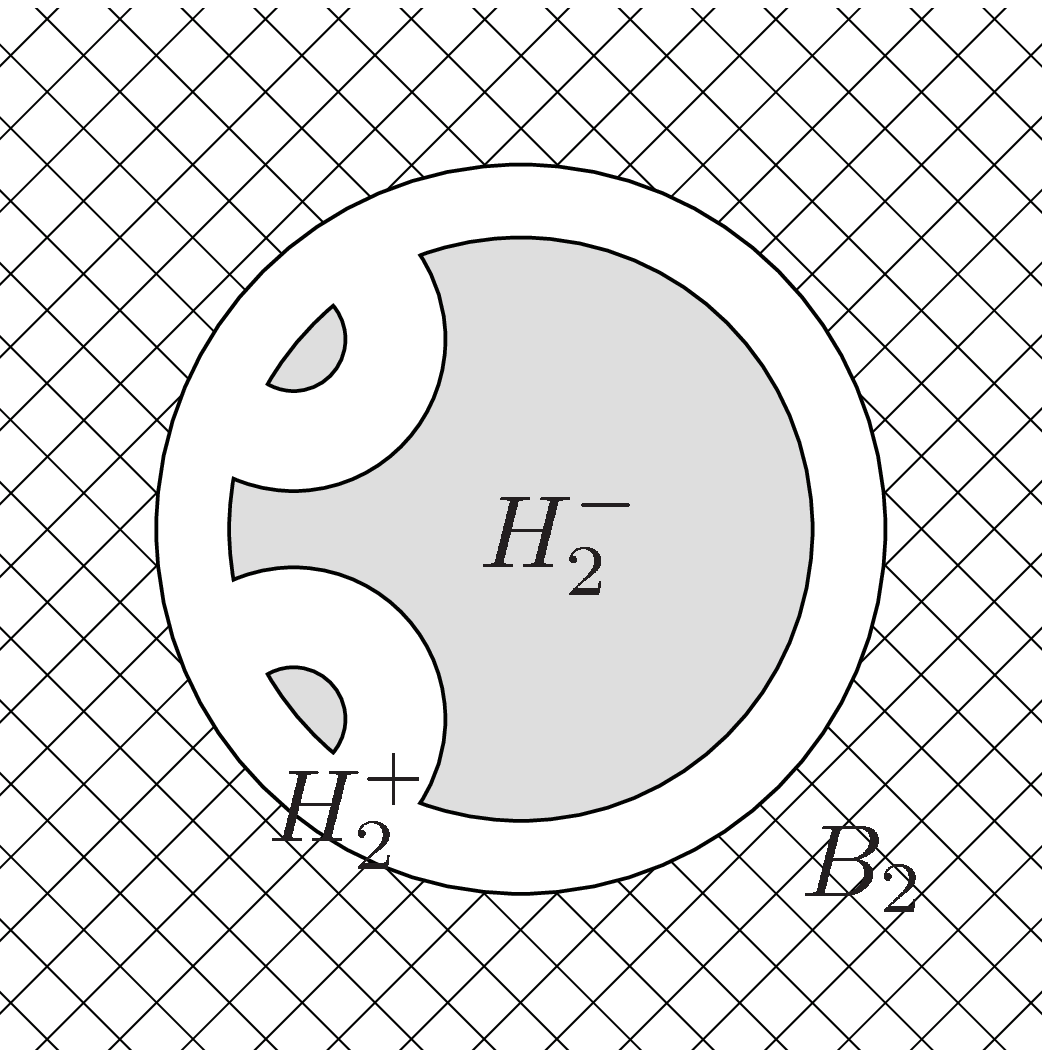}
\label{fig6}
\end{center}
\end{minipage}\\[0.2cm]
\begin{minipage}{3.8cm}
\begin{center}
\caption{}
\vspace{.2cm}
\includegraphics[width=2.5cm]{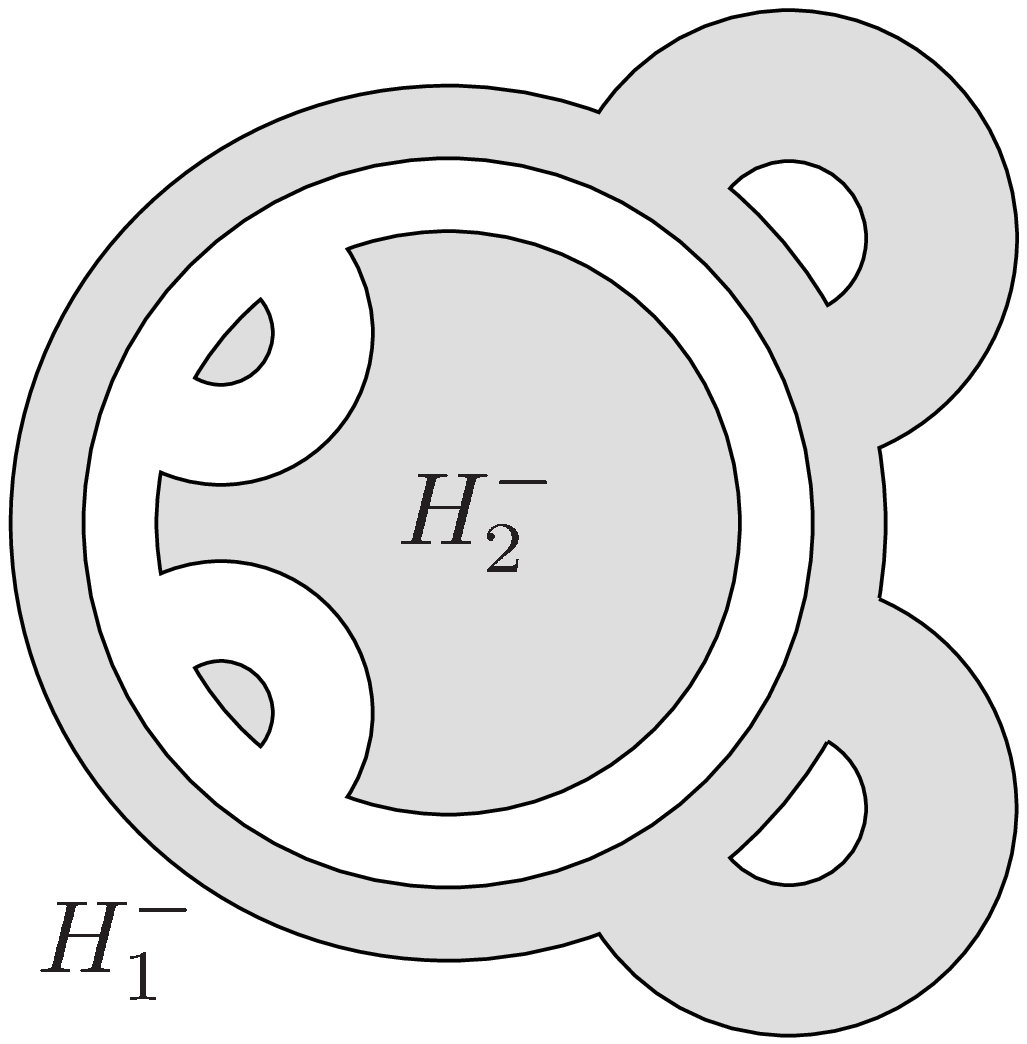}
\label{fig7}
\end{center}
\end{minipage}
\begin{minipage}{3.8cm}
\begin{center}
\caption{}
\vspace{.2cm}
\includegraphics[width=2.5cm]{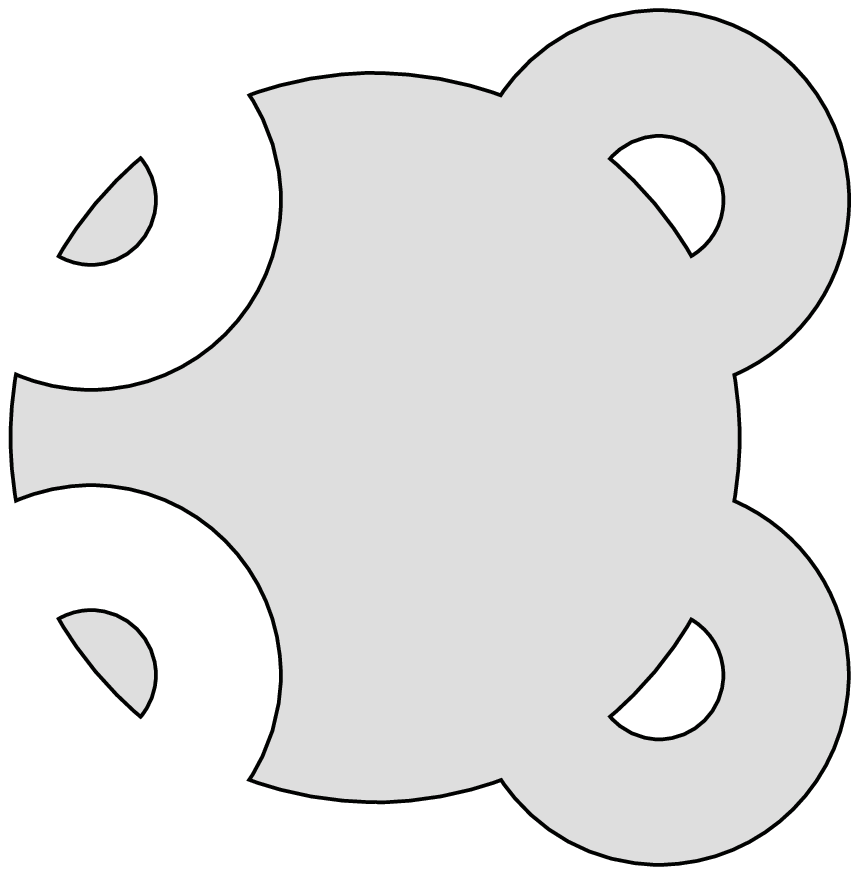}
\label{fig8}
\end{center}
\end{minipage}\\[0.2cm]
\begin{minipage}{3.8cm}
\begin{center}
\caption{}
\vspace{.2cm}
\includegraphics[width=2.5cm]{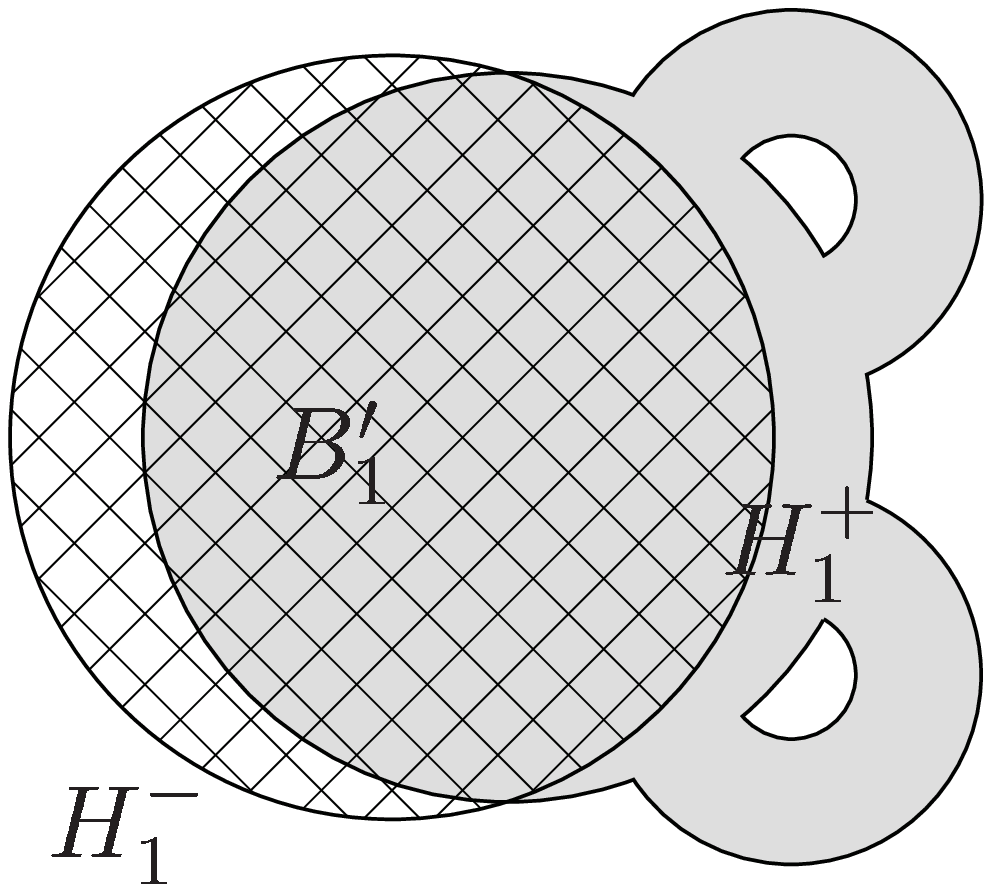}
\label{fig9}
\end{center}
\end{minipage}
\begin{minipage}{3.8cm}
\begin{center}
\caption{}
\vspace{.2cm}
\includegraphics[width=2.5cm]{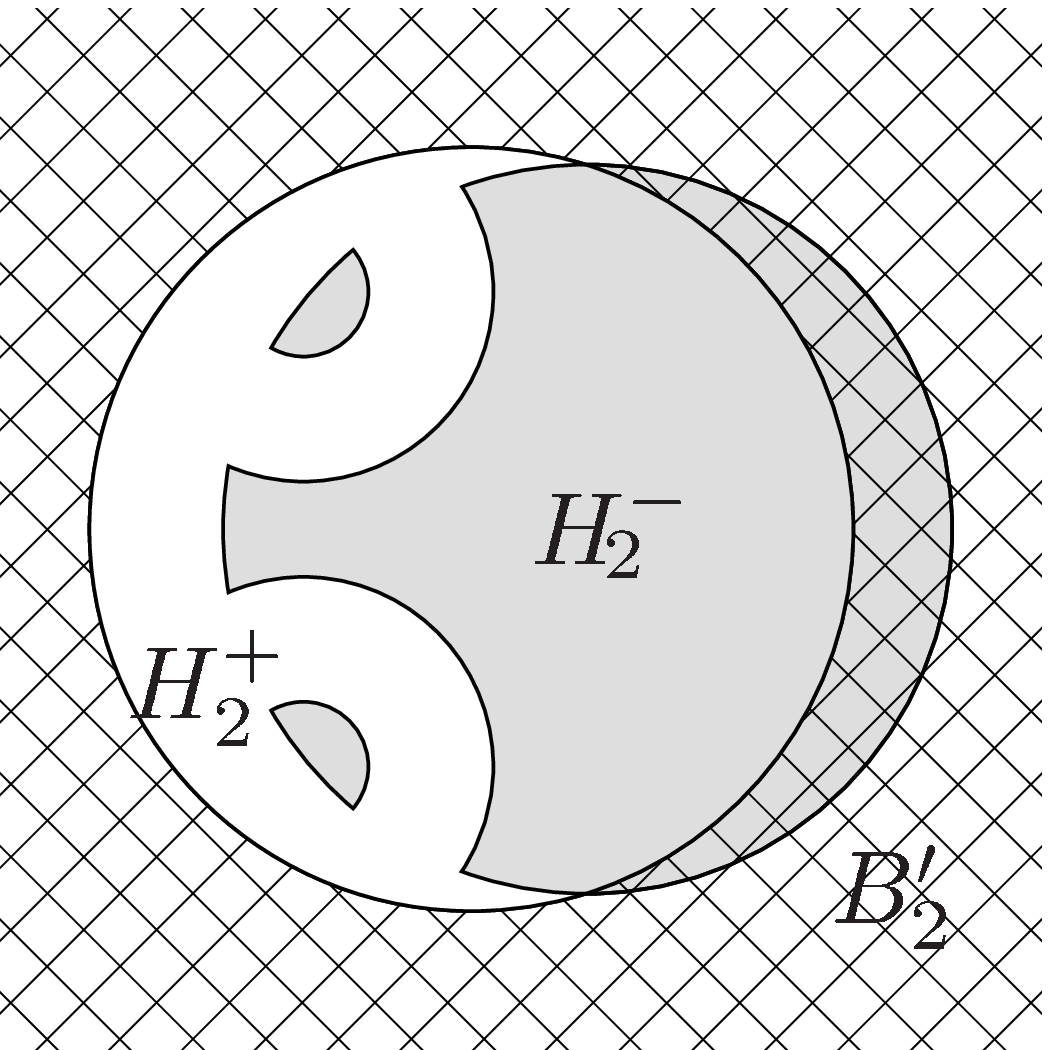}
\label{fig10}
\end{center}
\end{minipage}
\begin{minipage}{3.8cm}
\begin{center}
\caption{}
\vspace{.2cm}
\includegraphics[width=2.5cm]{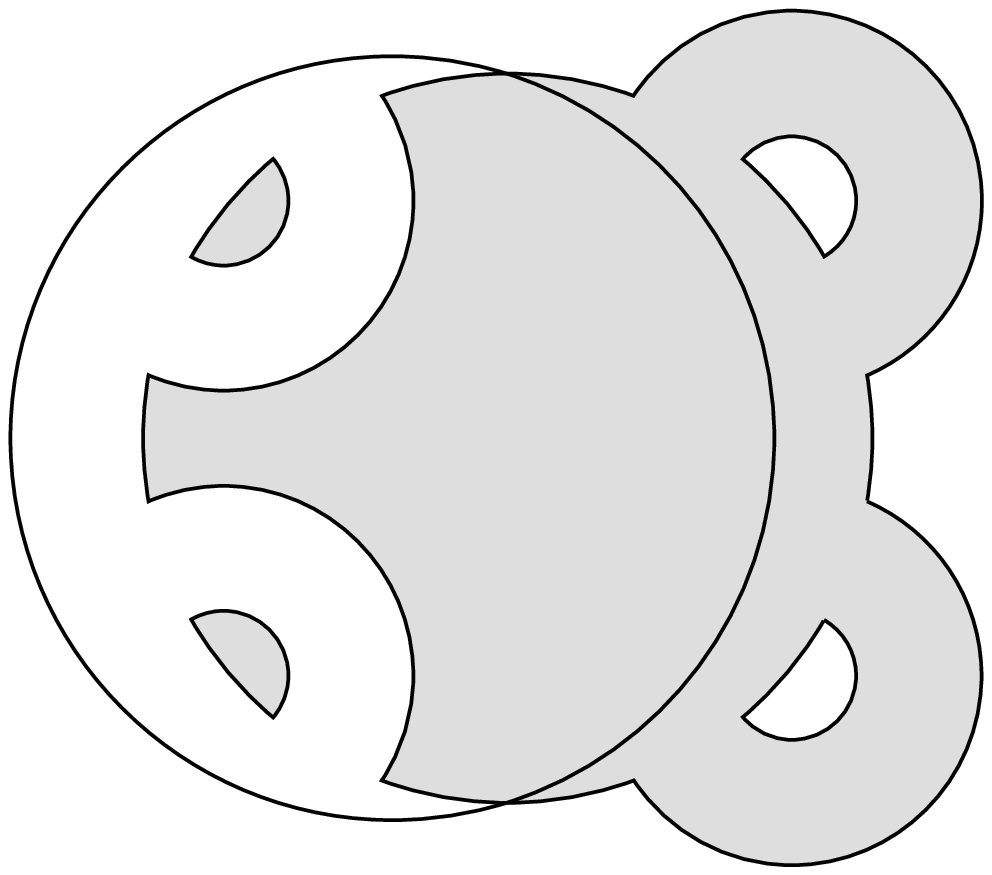}
\label{fig11}
\end{center}
\end{minipage}
\end{tabular}\\
\end{figure}

\begin{Prop}\label{prop2}
Suppose $(\Sigma ,H^-,H^+)$ is a Heegaard splitting for a closed\break $3$-manifold $M$, and $B^-,B^+$ are open balls in $H^-,H^+$, respectively.
Then the boundary stabilization of $(\Sigma ,H^-,H^+\setminus B^+)$ is isotopic (in the oriented version) to $(\Sigma ,H^+,H^-\setminus B^-)$.
\end{Prop}

This can be proved by pushing $B^+$ into $H^-$ from $H^+$.
The details are left to the reader.

\section{Sweep-outs and graphics}\label{Sweep-outs}

Rubinstein and Scharlemann \cite{rubinstein1} introduced a powerful machinery to analyze Heegaard splittings.
It is called the {\it Rubinstein-Scharlemann graphic} or just the {\it graphic} for short.
Roughly speaking, it is a $1$-complex in $[-1,1]\times [-1,1]$ representing the relation between two Heegaard splittings for a $3$-manifold.
While their original construction was based on the Cerf theory \cite{cerf}, it is useful to define it in terms of {\it stable} maps after Kobayashi and Saeki \cite{kobayashi2}.

Suppose $X,Y$ are smooth manifolds and $\varphi ,\psi :X\rightarrow Y$ are smooth maps.
The maps $\varphi $ and $\psi $ are called {\it isotopic} if there are diffeomorphisms $h_X:X\rightarrow X$ and $h_Y:Y\rightarrow Y$, each isotopic to the identity map on its respective space, such that $\varphi =h_Y\circ \psi \circ h_X$.
A smooth map $\varphi :X\rightarrow Y$ is called stable if there exists an open neighborhood $U$ of $\varphi $ in $C^\infty (X,Y)$ (under the Whitney $C^\infty $ topology, see \cite{hirsch}) such that every map in $U$ is isotopic to $\varphi $.
A Morse function is a stable function from a smooth manifold to $\mathbb{R}$.

Suppose $M$ is a compact, orientable, connected, smooth $3$-manifold, and $\partial M=\partial _-M\sqcup\partial _+M$ is a partition of boundary components of $M$.
Let $\Theta ^-$ be a finite graph in $M$ adjacent to all components of $\partial _-M$ and let $\Theta ^+$ similarly for $\partial _+M$.
A {\it sweep-out} for $M$ is a smooth function $f:M\rightarrow [-1,1]$ such that $f^{-1}(t)$ is a closed, connected surface parallel to $f^{-1}(0)$ for $t\in (-1,1)$, while $f^{-1}(-1)=\Theta ^-\cup \partial _-M$ and $f^{-1}(1)=\Theta ^+\cup \partial _+M$.
The sets $\Theta ^-\cup \partial _-M$ and $\Theta ^+\cup \partial _+M$ are called the {\it spines} of $f$.
We will say that $f$ {\it represents} a Heegaard splitting $(\Sigma ,H^-,H^+)$ for $M$ if $f$ can be isotoped so that $f^{-1}(0)=\Sigma ,f^{-1}(-1)\subset H^-$ and $f^{-1}(1)\subset H^+$.

Suppose $M_i$ is a compact, orientable, connected, smooth, $3$-dimensional submanifold of a smooth $3$-manifold $M$, and $f_i$ is a sweep-out for $M_i$ for each $i=1,2$.
Assume $M_1\cap M_2$ is a non-empty $3$-dimensional submanifold of $M$.
We define a smooth map $f_1\times f_2:M_1\cap M_2\rightarrow [-1,1]\times [-1,1]$ by $(f_1\times f_2)(p)=(f_1(p),f_2(p))$.
In the case when $M_1=M_2=M$, Kobayashi and Saeki \cite{kobayashi2} showed that we can deform $f_1$ and $f_2$ by an arbitrarily small isotopy so that $f_1\times f_2$ is stable on the complement of the spines of $f_1$ and $f_2$.
An almost identical argument induces the same property in the general case.
Thus, we can assume $f_1\times f_2$ is a stable map on the complement $M^*$ of the spines of $f_1$ and $f_2$ in $M_1\cap M_2$.

The Rubinstein-Scharlemann graphic for $f_1$ and $f_2$ is a properly embedded $1$-complex in $[-1,1]\times [-1,1]$ naturally extended from the discriminant set of $(f_1\times f_2)\mid _{M^*}$.
We mean the discriminant set as the image of the singular set $S_{f_1\times f_2}=\{ p\in M^*\mid {\rm rank}(d(f_1\times f_2))_p\leq 1\} $.
The singular set $S_{f_1\times f_2}$ is a $1$-dimensional smooth submanifold in $M^*$ consisting of all the points where a level surface of $f_1$ is tangent to a level surface of $f_2$.
The tangent point is either a ``center" or a ``saddle".
The discriminant set is a smooth immersion of $S_{f_1\times f_2}$ into $(-1,1)\times (-1,1)$ with normal crossings except for finitely many cusps.
We regard the crossings as valence-four vertices and the cusps as valence-two vertices of the graphic.
They are called {\it crossing vertices} and {\it birth-death vertices}, respectively.
On the boundary of $[-1,1]\times [-1,1]$, there are valence-one or valence-two vertices of the graphic.
Each edge is monotonously increasing or decreasing as a graph in $(-1,1)\times (-1,1)$.
See \cite{kobayashi2} or \cite{rubinstein1} for detailed descriptions.

For each $s\in (-1,1)$, the pre-image in $f_1\times f_2$ of the vertical arc $\{ s\} \times [-1,1]$ is the level surface $f^{-1}_1(s)$.
The restriction of $f_2$ to the level surface has critical levels corresponding to the intersections of the vertical arc and the graphic.

\begin{Def}
Sweep-outs $f_1$ and $f_2$ are called {\it generic} if $f_1\times f_2$ is stable on $M^*$ and every vertical or horizontal arc on $[-1,1]\times [-1,1]$ contains at most one vertex of the graphic.
\end{Def}

\section{Labeling the graphics}

We will characterize some relations of the level surfaces of sweep-outs.
It gives a ``labeling" for the complementary regions of the graphic.
This kind of labeling is one of the most useful techniques for reading graphics.

Suppose $M$ is a compact, orientable, connected, smooth $3$-manifold, and $N$ is a $3$-dimensional submanifold of $M$.
Let $(\Sigma ,H^-,H^+)$ and $(T,G^-,G^+)$ be Heegaard splittings for $M$ and $N$, respectively.
Let $f$ and $g$ be sweep-outs representing $(\Sigma ,H^-,H^+)$ and $(T,G^-,G^+)$, respectively.
We will use the notations like $\Sigma _s=f^{-1}(s),H^-_s=f^{-1}([-1,s]),H^+_s=f^{-1}([s,1])$ and $T_t=g^{-1}(t)$.

\begin{Def}
For $s,t\in (-1,1)$, we will say that $T_t$ is {\it mostly above} $\Sigma _s$ if $H^-_s\cap T_t$ is contained in a disk in $T_t$.
Similarly, $T_t$ is {\it mostly below} $\Sigma _s$ if $H^+_s\cap T_t$ is contained in a disk in $T_t$.
\end{Def}

\begin{Def}
For generic sweep-outs $f$ and $g$, we will say that $f$ {\it spans} $g$ if $T_{t_-}$ is mostly below $\Sigma _s$ and $T_{t_+}$ is mostly above $\Sigma _s$ for some values $s,t_-,t_+\in (-1,1)$.
Moreover, we will say that $f$ spans $g$ {\it positively} if $t_-<t_+$, or {\it negatively} if $t_->t_+$.
\end{Def}

\begin{Def}
For generic sweep-outs $f$ and $g$, we will say that $f$ {\it splits} $g$ if there is a value $s\in (-1,1)$ such that for every $t\in (-1,1)$, the level surface $T_t$ is neither mostly above nor below $\Sigma _s$.
\end{Def}

Let $R_a$ be the set of points $(s,t)\in (-1,1)\times (-1,1)$ such that $T_t$ is mostly above $\Sigma _s$.
Similarly, let $R_b$ be the set of points such that $T_t$ is mostly below $\Sigma _s$.
Note that if a point $(s,t)$ is in $R_a$ then its left side $(-1,s]\times \{ t\} $ is contained in $R_a$ because the area $H^-_s\cap T_t$ in the surfaces $T_t$ increase with $s$.
Symmetrically, if $(s,t)\in R_b$ then $[s,1)\times \{ t\} \subset R_b$.
The right side of $R_a$ and the left side of $R_b$ are bounded by edges of the graphic.

Figure \ref{fig1} illustrates the condition that $f$ spans $g$ positively.
In Figure \ref{fig2}, $f$ spans $g$ negatively.
In Figure \ref{fig3}, $f$ spans $g$ positively and negatively.
In Figure \ref{fig4}, $f$ splits $g$.
Note that exactly one of the conditions spanning or splitting happens for any generic pair of sweep-outs.

\begin{figure}[ht]
\begin{tabular}{cc}
\begin{minipage}{3.8cm}
\begin{center}
\caption{}
\vspace{.2cm}
\includegraphics[width=3cm]{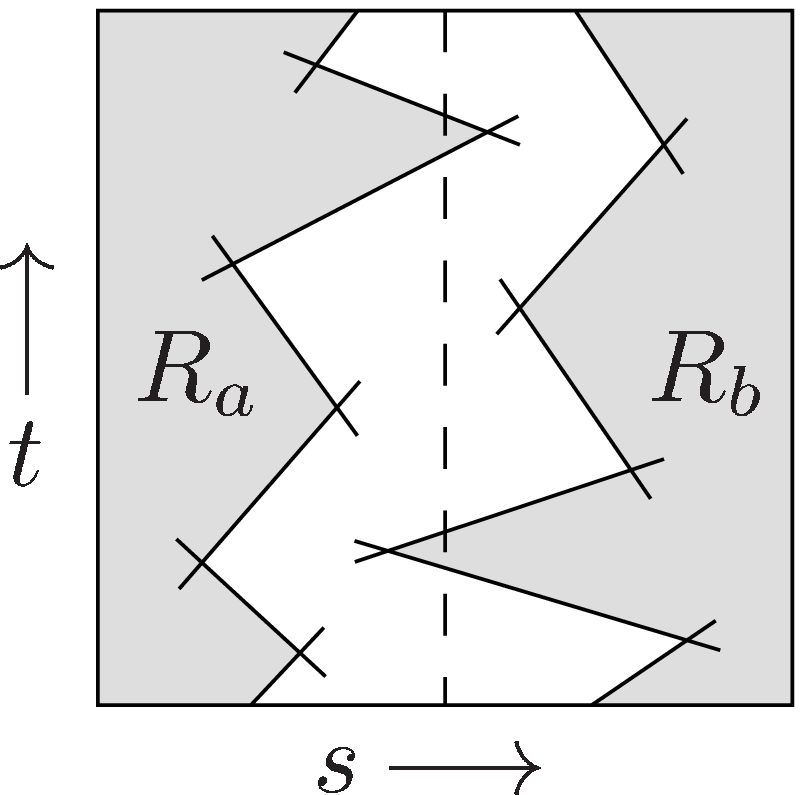}
\label{fig1}
\end{center}
\end{minipage}
\begin{minipage}{3.8cm}
\begin{center}
\caption{}
\vspace{.2cm}
\includegraphics[width=3cm]{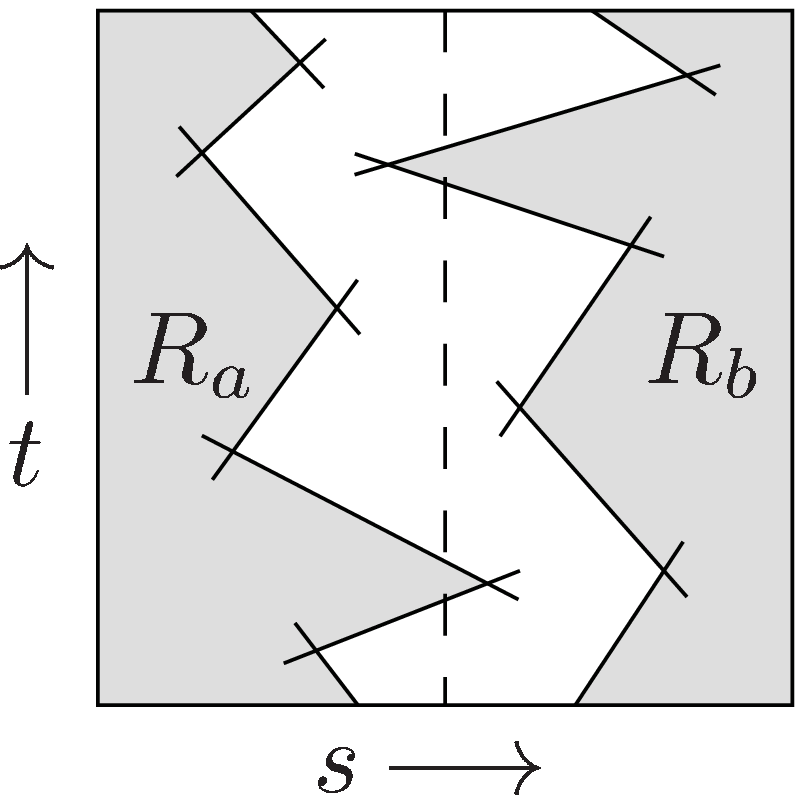}
\label{fig2}
\end{center}
\end{minipage}\\[0.2cm]
\begin{minipage}{3.8cm}
\begin{center}
\caption{}
\vspace{.2cm}
\includegraphics[width=3cm]{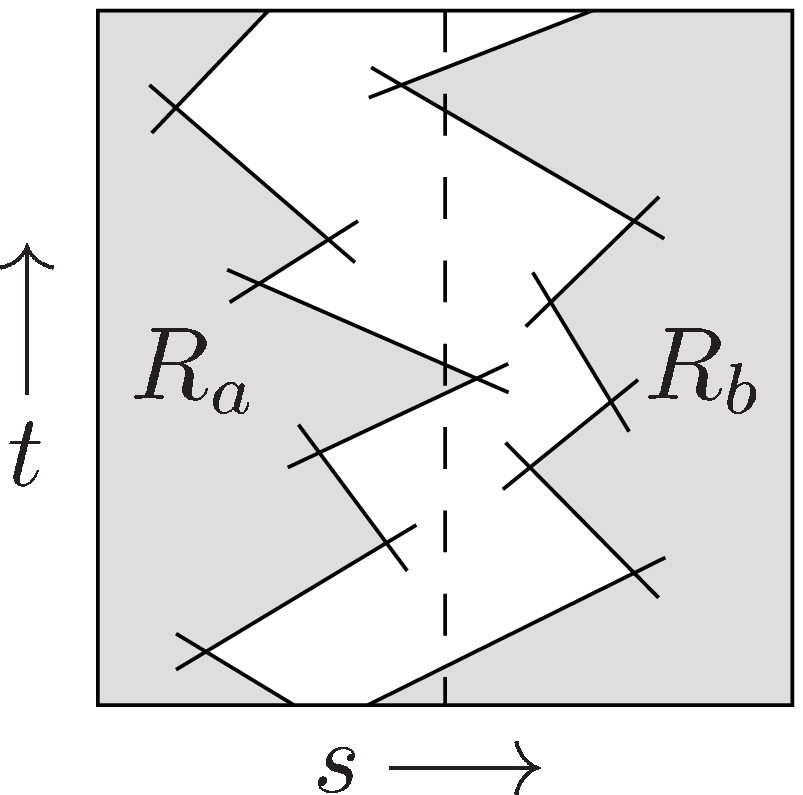}
\label{fig3}
\end{center}
\end{minipage}
\begin{minipage}{3.8cm}
\begin{center}
\caption{}
\vspace{.2cm}
\includegraphics[width=3cm]{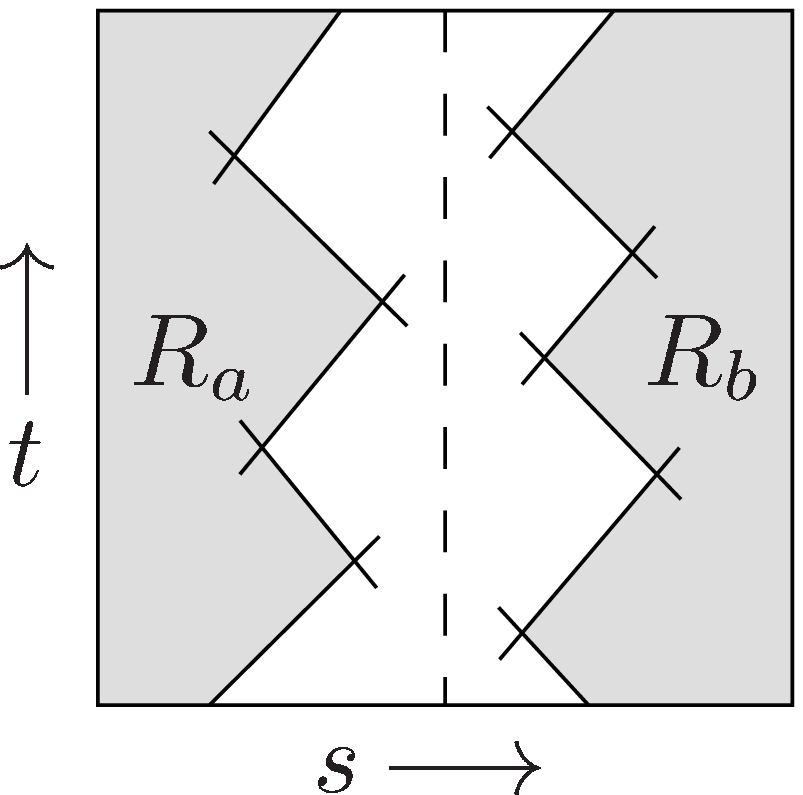}
\label{fig4}
\end{center}
\end{minipage}
\end{tabular}
\end{figure}

\begin{Def}
We will say that $(\Sigma ,H^-,H^+)$ {\it spans} $(T,G^-,G^+)$ positively (negatively) if $(\Sigma ,H^-,H^+)$ and $(T,G^-,G^+)$ are represented by generic sweep-outs $f$ and $g$, respectively, such that $f$ spans $g$ positively (negatively).
We will also say that $(\Sigma ,H^-,H^+)$ {\it splits} $(T,G^-,G^+)$ if $(\Sigma ,H^-,H^+)$ and $(T,G^-,G^+)$ are represented by generic sweep-outs $f$ and $g$ such that $f$ splits $g$.
\end{Def}

Note that if $(\Sigma ,H^-,H^+)$ spans $(T,G^-,G^+)$ positively, $(\Sigma ,H^+,H^-)$ spans\break $(T,G^-,G^+)$ negatively.

\section{Spanning sweep-outs}\label{SpanningS}

The spanning condition gives a bound for the genus of one of the Heegaard splittings.
Suppose $(\Sigma ,H^-,H^+)$ is a Heegaard splitting for a smooth $3$-manifold $M$, and $(T,G^-,G^+)$ is a Heegaard splitting for a $3$-dimensional submanifold $N$ of $M$.
Suppose $f$ and $g$ are generic sweep-outs representing $(\Sigma ,H^-,H^+)$ and $(T,G^-,G^+)$, respectively.
Assume $f$ spans $g$ positively.

\begin{figure}[ht]
\begin{center}
\caption{}
\vspace{.2cm}
\includegraphics[width=5.1cm]{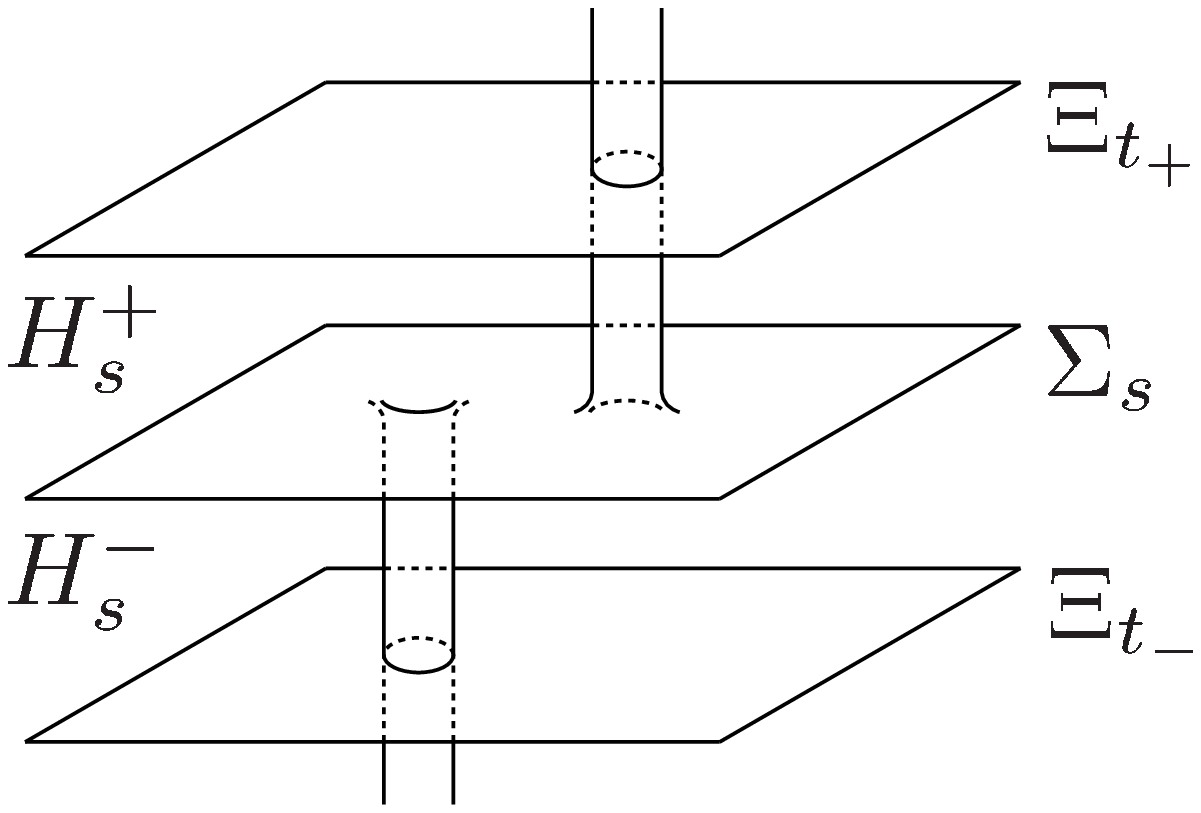}
\label{fig16}
\end{center}
\end{figure}

By the definition, there is a value $-1<s<1$ and values $-1<t_-<t_+<1$ such that $T_{t_-}$ is mostly below $\Sigma _s$ and $T_{t_+}$ is mostly above $\Sigma _s$.
That is to say, $T_{t_-}$ is contained in $H^-_s$ except for some disks while $T_{t_+}$ is contained in $H^+_s$ except for some disks as Figure \ref{fig16}.
In the product manifold $g^{-1}([t_-,t_0])$, the surface $\Sigma _s$ must be ``mostly separating" one boundary component from the other.
The reader can notice that $\Sigma _s\cap g^{-1}([t_-,t_+])$ has genus at least the genus of $T$.
By similar observations, we have the following:

\begin{Lem}\label{spanning}
If $f$ spans $g$ then $\Sigma _s \cap N$ has genus at least the genus of $T$ for some value $s\in (-1,1)$.
If $f$ spans $g$ positively and negatively then $\Sigma _s \cap N$ has genus at least twice the genus of $T$ for some value $s\in (-1,1)$.
\end{Lem}

Recall that we allow $3$-manifolds to have sphere boundaries.
Still, next four lemmas can be proved identically as those in brackets.

\begin{Lem}\cite[Lemma 9]{johnson1}\label{1-9}
Every Heegaard splitting spans itself positively.
\end{Lem}

\begin{Lem}\cite[Lemma 12]{johnson2}\label{2-12}
If $(\Sigma ,H^-,H^+)$ spans $(T,G^-,G^+)$ positively (negatively) then every stabilization of $(\Sigma ,H^-,H^+)$ spans $(T,G^-,G^+)$ positively (negatively).
\end{Lem}

\begin{Lem}\cite[Lemma 14]{johnson2}\label{2-14}
Suppose $(\Sigma _1,H^-_1,H^+_1)$ and $(\Sigma _2,H^-_2,H^+_2)$ are Heegaard splittings for compact, smooth $3$-manifolds $M_1$ and $M_2$, respectively.
Let $(\Sigma ,H^-,H^+)$ be the amalgamation of $(\Sigma _1,H^-_1,H^+_1)$ and $(\Sigma _2,H^-_2,H^+_2)$.
Suppose $(T,G^-,G^+)$ is a Heegaard splitting for a $3$-dimensional submanifold $N$ of $M_1$.
If $(\Sigma _1,H^-_1,H^+_1)$ spans $(T,G^-,G^+)$ positively (negatively) then $(\Sigma ,H^-,H^+)$ spans $(T,G^-,G^+)$ positively (negatively).
\end{Lem}

\begin{Lem}\cite[Lemma 16]{johnson2}\label{2-16}
Suppose $M$ is a smooth $3$-manifold with a single boundary component and $(\Sigma ,H^-,H^+)$ is a Heegaard splitting for $M$ such that $\partial _-H^+=\partial M$.
Suppose $(T,G^-,G^+)$ is a Heegaard splitting for a $3$-dimensional submanifold $N$ of $M$.
Let $(\Sigma ',H^{\prime -},H^{\prime +})$ be the boundary stabilization of $(\Sigma ,H^-,H^+)$.
If $(\Sigma ,H^-,H^+)$ spans $(T,G^-,G^+)$ positively (negatively) then $(\Sigma ',H^{\prime -},H^{\prime +})$ spans $(T,G^-,G^+)$ negatively (positively).
\end{Lem}

\section{Splitting sweep-outs}\label{SplittingS}

The {\it curve complex} $C(T)$ of a closed, orientable, connected surface $T$ is a simplicial complex defined as follows:
The vertices of $C(T)$ are isotopy classes of essential loops in $T$.
Distinct $n$ vertices span a $(n-1)$-simplex of $C(T)$ if and only if they are represented by pairwise disjoint loops in $T$.
There is a canonical distance $d$ among the vertices.
We mean that $d(v_1,v_2)$ is the number of edges on the shortest path between two vertices $v_1$ and $v_2$ in the $1$-skeleton of $C(T)$. 

Suppose $(T,G^-,G^+)$ is a Heegaard splitting.
When $D^-$ and $D^+$ are essential disks in $G^-$ and $G^+$, respectively, $\partial D^-$ and $\partial D^+$ can be regarded as vertices of $C(T)$.
Hempel \cite{hempel} defined the {\it distance} of $(T,G^-,G^+)$, denoted by $d(T)$, as the minimum of $d(\partial D^-$, $\partial D^+)$ over all pairs of essential disks $D^-\subset G^-,D^+\subset G^+$.
It is a numerical invariant indicating the irreducibility of Heegaard splittings (see \cite{hempel}).

The goal in this section is to estimate the genus of $(\Sigma ,H^-,H^+)$ by $d(T)$ when a Heegaard splitting $(\Sigma ,H^-,H^+)$ splits another Heegaard splitting $(T,G^-,G^+)$.
We will almost trace the way of \cite[Section 6]{johnson2} but modify it slightly to avoid arguments with the irreducibility of the manifolds.

Suppose $M_1$ and $M_2$ are irreducible, closed, smooth $3$-manifolds other than $S^3$.
Let $M^*_i$ be the $3$-manifold obtained by removing an open ball from $M_i$ for each $i=1,2$.
Let $M$ be the union of $M^*_1$ and $M^*_2$ glued at their boundaries, namely, the connected sum of $M_1$ and $M_2$.
Take either $M^*_1$ or $M^*_2$, and rewrite it as $N$.
Suppose $(\Sigma ,H^-,H^+)$ is a Heegaard splitting of genus $k$ for $M$, and $(T,G^-,G^+)$ is a Heegaard splitting of genus at least $2$ with distance at least $2$ for $N$.
Assume $(\Sigma ,H^-,H^+)$ splits $(T,G^-,G^+)$.
By definition, there are generic sweep-outs $f$ and $g$ representing $(\Sigma ,H^-,H^+)$ and $(T,G^-,G^+)$, respectively such that $f$ splits $g$.

\begin{Lem}\label{2-21}
There exists a value $s_0\in (-1,1)$ such that:
\begin{enumerate}
\item There are no vertices of the graphic on the vertical arc $\{ s_0\} \times [-1,1]$.
\item $\Sigma _{s_0}\cap T_t$ contains an essential loop in $T_t$ for each regular value $t$ for $g\mid _{\Sigma _{s_0}}$.
\end{enumerate}
\end{Lem}

\begin{proof}
Let $C$ be the set of values $s_0\in (-1,1)$ satisfying the condition (2).
When the condition (2) fails, either $H^-_{s_0}\cap T_t$ or $H^+_{s_0}\cap T_t$ is contained in a disk in $T_t$ for some value $t$, so $T_t$ is mostly above or below $\Sigma _{s_0}$.
Therefore $C$ can be considered as the complement of the projections of $R_a\cup R_b$ in $[-1,1]\times \{ \text{pt}\} $.
Since $f$ splits $g$, the set $C$ is a non-empty closed interval.

If $C$ is a single point $\{ s_1\} $, there is a crossing vertex $(s_1,t_1)$ of which the left quadrant is contained in $R_a$ and the right quadrant is contained in $R_b$.
For a small $\varepsilon $, the intersection $H^+_{s_1-\varepsilon }\cap T_{t_1}$ becomes $H^+_{s_1+\varepsilon }\cap T_{t_1}$ by a transformation including only two singularities.
However, $H^+_{s_1-\varepsilon }\cap T_{t_1}$ is contained in a disk while $H^+_{s_1+\varepsilon }\cap T_{t_1}$ covers $T_{t_1}$ except for some disks.
This is possible only when $T_{t_1}$ is a torus.
Since we assume the genus of $(T,G^-,G^+)$ is at least $2$, the closed interval $C$ is non-trivial.

There are finitely many vertices in the graphic, so there exists a value $s_0$ in $C$ such that the vertical arc $\{ s_0\} \times [-1,1]$ passes through no vertices of the graphic.
\end{proof}

Similarly to $H^-_s$ and $H^+_s$, we will write $G^-_t=g^{-1}([-1,t])$ and $G^+_t=g^{-1}([t,1])$.

\begin{Lem}\label{schleimer}
There exists a non-trivial closed interval $[a,b]\subset [-1,1]$ such that:
\begin{enumerate}
\item For a small $\varepsilon $, the intersection $\Sigma _{s_0}\cap T_{a-\varepsilon}$ has a component bounding an essential disk of $G^-_{a-\varepsilon }$ or $a=-1$.
\item For each $t\in (a,b)$, the intersection $\Sigma _{s_0}\cap T_t$ does not have any loops bounding essential disks of $G^-_t$ or $G^+_t$.
\item For a small $\varepsilon $, the intersection $\Sigma _{s_0}\cap T_{b+\varepsilon}$ has a component bounding an essential disk of $G^+_{b+\varepsilon }$ or $b=1$.
\end{enumerate}
\end{Lem}

\begin{proof}
Let $R_-$ be the set of points $(s,t)\in (-1,1)\times (-1,1)$ such that $\Sigma _s\cap T_t$ has a component bounding an essential disk of $G^-_t$.
Similarly, Let $R_+$ be the set of points such that $\Sigma _s\cap T_t$ has a component bounding an essential disk of $G^+_t$.
They determine another labeling for the graphic.

Let $a$ be the maximum of the closure of $R_-\cap (\{ s_0\} \times [-1,1])$ (or $-1$ if $R_-\cap (\{ s_0\} \times [-1,1])=\emptyset $).
Let $b$ be the minimum of the closure of $R_+\cap (\{ s_0\} \times [a,1])$ (or $1$ if $R_+\cap (\{ s_0\} \times [a,1])=\emptyset $).

If there is a horizontal arc $[-1,1]\times \{ t_0\} $ which intersects both $R_-$ and $R_+$, the level surface $T_{t_0}$ has a level loop of $f\mid _{T_{t_0}}$ bounding an essential disk of $G^-_t$ and a level loop bounding an essential disk of $G^+_t$.
It contradicts that the distance of $(T,G^-,G^+)$ is at least $2$.
Therefore no horizontal arcs intersect both $R_-$ and $R_+$.
If $a=b$ then $(s_0,a)$ must be a crossing vertex of the graphic.
Since there are no vertices on $\{ s_0\} \times [-1,1]$, the closed interval $[a,b]$ is non-trivial.
\end{proof}

Figure \ref{fig12} illustrates the segment $\{ s_0\} \times [a,b]$.
We will consider the intersection loops on this segment and construct a subcomplex of $C(T_0)$ from these loops.

\begin{figure}[ht]
\begin{center}
\caption{}
\vspace{.2cm}
\includegraphics[width=3.3cm]{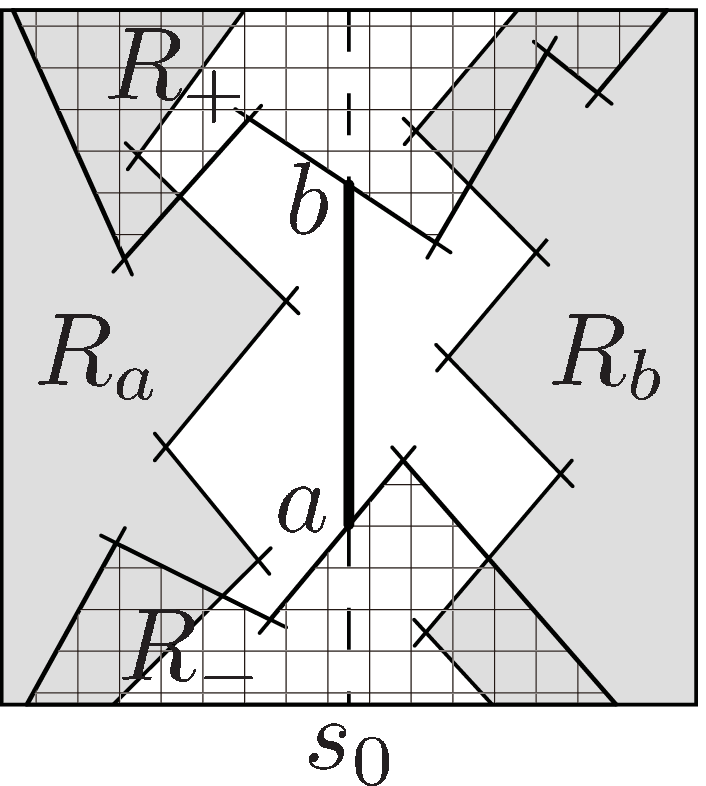}
\label{fig12}
\end{center}
\end{figure}

Let $a'$ be a regular value for $g\mid _{\Sigma _{s_0}}$ just above $a$ and let $b'$ be a regular value for $g\mid _{\Sigma _{s_0}}$ just below $b$.
Let $\Delta $ be the union of the disks bounded by the inessential loops of $\Sigma _{s_0}\cap g^{-1}(\{ a',b'\} )$ in $\Sigma _{s_0}$.
Let $F$ be the union of $\Sigma _{s_0}\cap g^{-1}([a',b'])$ and $\Delta $.
Consider a projection map $\pi $ from $g^{-1}([a',b'])$ onto $T_0$.

\begin{Lem}
If two level loops of $g\mid _{F}$ are isotopic in $F$ then their projections are isotopic in $T_0$.
\end{Lem}

\begin{proof}
Any two level loops are disjoint in $F$ so if two level loops are isotopic then they bound an annulus $A\subset F$.
Note that $A$ may contain some disks of $\Delta $.
By the condition (2) in Lemma \ref{schleimer}, the boundary of a disk of $\Delta $ also bounds a disk in $T_{a'}$ or $T_{b'}$.
Replacing the disks of $\Delta $ by the disks in $T_{a'}$ or $T_{b'}$, we can produce a new annulus $A'$ contained in $g^{-1}([a',b'])$.
The projection of $A'$ into $T_0$ determines a homotopy from the image of one boundary of $A'$ to the image of the other.
Thus the projections of the two loops are isotopic.
\end{proof}

Let $L$ be the set of isotopy classes of level loops of $g\mid _{F}$.
A representative of an element $l\in L$ projects to a simple closed curve in $T_0$.
If the projection is essential in $T_0$, we define $\pi _*(l)$ to be the corresponding vertex of the curve complex $C(T_0)$.
If the projection is inessential, we define $\pi _*(l)=0$.
By the previous lemma, $\pi _*$ is well defined as a map from $L$ to the disjoint union $C(T_0)\sqcup \{ 0\} $.

Isotopy classes of essential level loops of $g\mid _{F}$ determine a pair-of-pants decomposition for $F$.
The following can be proved identically as \cite[Lemma 23]{johnson2}.

\begin{Lem}
If $l_1$ and $l_2$ are cuffs of the same pair of pants in $F\setminus L$ then their projections can be isotoped to be disjoint.
\end{Lem}

For each regular value $t\in [a',b']$ for $g\mid _{F}$, let $L^t$ be the set of isotopy classes of loops in $F\cap T_t$.
Loops in $F\cap T_t$ are pairwise disjoint so their projections are pairwise disjoint.
Moreover the projections contain at least one essential loop by the condition (2) in Lemma \ref{2-21}.
Therefore the subcomplex $L^t_C$ of $C(T_0)$ spanned by $\pi _*(L^t)\cap C(T_0)$ is non-empty.

If there are no critical levels for $g\mid _{F}$ between regular values $t_1$ and $t_2$ then $L^{t_1}=L^{t_2}$, so $L^{t_1}_C=L^{t_2}_C$.
If there is a single critical level of center tangency between $t_1$ and $t_2$, the difference between $L^{t_1}$ and $L^{t_2}$ is the isotopy class of a trivial loop in $F$.
By the condition (2) in Lemma \ref{schleimer}, a trivial loop in $F$ projects to a trivial loop in $T_0$.
It implies $\pi _*(L^{t_1})\cap T_0=\pi _*(L^{t_2})\cap T_0$, so $L^{t_1}_C=L^{t_2}_C$.
If there is a single critical level of saddle tangency between $t_1$ and $t_2$, either one loop in $F\cap T_{t_1}$ is replaced by two loops in $F\cap T_{t_2}$ or two loops in $F\cap T_{t_1}$ is replaced by one loop in $F\cap T_{t_2}$ at the critical level.
If those three loops are essential in $F$, they bound a pair of pants in $F\setminus L$.
By the previous lemma, their projections can be isotoped to be pairwise disjoint.
Thus, there is an edge of $C(T_0)$ connecting $L^{t_1}_C$ and $L^{t_2}_C$.
If one of those three loops is trivial in $F$ then $L^{t_1}_C$ and $L^{t_2}_C$ have common vertices.
Because $L$ is the union of $L^t$ over all regular values  for $g\mid _{F}$, the subcomplex $L_C$ of $C(T_0)$ spanned by $\pi _*(L)\cap C(T_0)$ is connected.

Consider two vertices $v$ and $v'$ in $L_C$.
Suppose $v=v_0,v_1,\ldots ,v_n=v'$ is the shortest edge path connecting them in $L_C$.
Let $l_i\in L$ projects to $v_i$ for each $i=0,1,\ldots ,n$.
If $l_i$ and $l_j$ are cuffs of the same pair of pants in $F\setminus L$ then there is an edge of $L_C$ connecting $v_i$ and $v_j$.
Since the path is minimal, $i$ and $j$ must be consecutive.
Then, we can estimate the diameter of $L_C$ by the number of pairs of pants in $F\setminus L$.
The number of pairs of pants in $F\setminus L$ is at most the negative Euler characteristic of $F$.
Since the boundary components of $F$ are essential in $\Sigma _{s_0}$, the Euler characteristic of $F$ is at least that of $\Sigma _{s_0}$.
We can conclude that the diameter of $L_C$ is at most $2k-2$.
See the proof of \cite[Lemma 24]{johnson2} for the details of this argument.

We are ready to prove the following:

\begin{Lem}\label{splitting}
If $(\Sigma ,H^-,H^+)$ splits $(T,G^-,G^+)$ then $2k\geq d(T_0)$.
\end{Lem}

\begin{proof}
Consider the case $a>-1$.
By the condition (1) and (2) in Lemma \ref{schleimer}, $\Sigma _{s_0}\cap T_{a-\varepsilon}$ has a component bounding an essential disk of $G^-_{a-\varepsilon }$ while $\Sigma _{s_0}\cap T_{a+\varepsilon}$ does not.
That implies $a$ must be a critical level for $g\mid _{\Sigma _{s_0}}$ containing a saddle tangency.
As above, the projections of the level loops before and after this singularity can be isotoped to be pairwise disjoint.
The projection of one of the level loops before this singularity bounds an essential disk of $G^-_0$.
The projections of the level loops after this singularity are contained in $L_C$.
Thus, the boundary of the essential disk of $G^-_0$ is connected to $L_C$ by an edge in $C(T_0)$.

Consider the case $a=-1$.
The compression body $G^-_{a'}$ is a small neighborhood of the spine.
If $G^-_{a'}$ is a handlebody, every component of $\Sigma _{s_0}\cap G^-_{a'}$ is an essential disk of $G^-_{a'}$.
It contradicts the condition (2) in Lemma \ref{schleimer}.
Therefore $\partial _-G^-_{a'}=\partial N$ and every component of $\Sigma _{s_0}\cap T_{a'}$ is parallel to $\partial _-G^-_{a'}$.
The compression body $G^-_{a'}$ has essential disks disjoint from any such loop because the genus of $\partial _+G^-_{a'}$ is at least $2$. 
Similarly to the above argument, the boundary of an essential disk of $G^-_0$ is connected to $L_C$ by an edge in $C(T_0)$.

Symmetrical arguments for $b$ imply that the boundary of an essential disk of $G^+_0$ is connected to $L_C$ by an edge in $C(T_0)$.
Since the diameter of $L_C$ is at most $2k-2$, the distance of $(T,G^-,G^+)$ is at most $2k$.
\end{proof}

\section{Isotopies of sweep-outs}\label{Isotopy}

While we recognize Heegaard splittings up to isotopy, the spanning or splitting condition can be changed by isotopies of the sweep-outs.
In this section, we need to observe the transition of the condition during an isotopy of one of the sweep-outs.
Recall we defined isotopies of smooth maps in Section \ref{Sweep-outs}.

Suppose again $M_1$ and $M_2$ are irreducible, closed, smooth $3$-manifolds other than $S^3$.
Let $M^*_i$ be the $3$-manifold obtained by removing an open ball from $M_i$ for each $i=1,2$.
Let $M$ be the union of $M^*_1$ and $M^*_2$ glued at their boundaries.
Take either $M^*_1$ or $M^*_2$, and rewrite it as $N$.
Suppose $(\Sigma ,H^-,H^+)$ is a Heegaard splitting for $M$, and $(T,G^-,G^+)$ is a Heegaard splitting of genus at least $2$ for $N$.

\begin{Lem}\label{isotopy}
If $(\Sigma ,H^-,H^+)$ spans $(T,G^-,G^+)$ positively and negatively then either there is a pair of sweep-outs $f$ and $g$ representing $(\Sigma ,H^-,H^+)$ and $(T,G^-,G^+)$ such that $f$ spans $g$ positively and negatively or $(\Sigma ,H^-,H^+)$ splits $(T,G^-,G^+)$.
\end{Lem}

\begin{proof}
Since $(\Sigma ,H^-,H^+)$ spans $(T,G^-,G^+)$ positively, there are generic sweep-outs $f_0$ and $g$ representing $(\Sigma ,H^-,H^+)$ and $(T,G^-,G^+)$, respectively such that $f_0$ spans $g$ positively.
Since $(\Sigma ,H^-,H^+)$ also spans $(T,G^-,G^+)$ negatively, there are generic sweep-outs $f'$ and $g'$ representing $(\Sigma ,H^-,H^+)$ and $(T,G^-,G^+)$, respectively such that $f'$ spans $g'$ negatively.

The sweep-outs $g$ and $g'$ represent the same Heegaard splitting, so $g'$ will be isotopic to $g$ after an appropriate sequence of handle slides of the spines.
The handle slides can be done in an arbitrarily small neighborhood of the original spines so that $f'$ still spans $g'$ negatively.
Therefore we can assume there is an isotopy taking $g'$ to $g$.
By the definition, there are diffeomorphisms $h_N:N\rightarrow N$ and $h_I:[-1.1]\rightarrow [-1,1]$ such that $g=h_I\circ g'\circ h_N$.
Let $h_M:M\rightarrow M$ be an arbitrary extension of $h_N$, and define $f_1=h_I\circ f'\circ h_M$.
Then $f_1$ spans $g$ negatively.

Similarly, we can assume $f_0$ is isotopic to $f_1$ because $f_0$ and $f_1$ represent the same Heegaard splitting.
According to \cite[Lemma 26]{johnson2}, there is a continuous family of sweep-outs $\{ f_r\mid r\in [0,1]\} $ such that $f_r$ and $g$ is generic for all but finitely many $r\in [0,1]$.
At the finitely many non-generic points, there are at most two valence-two or valence-four vertices at the same level, or one valence-six vertex.

For a generic value $r$, the sweep-out $f_r$ either spans $g$ or splits $g$.
Then we can assume that except for finitely many non-generic values, $f_r$ spans $g$ positively or negatively, but not both.
Since $f_0$ spans $g$ positively and $f_1$ spans $g$ negatively, there must be some non-generic value $r_0$ such that $f_{r_0-\varepsilon }$ spans $g$ positively while $f_{r_0+\varepsilon }$ spans $g$ negatively for a small $\varepsilon >0$.
Then we may consider three cases like Figures \ref{fig13}, \ref{fig14} and \ref{fig15}.
In the case Figure \ref{fig13} or \ref{fig14}, there are three valence-four vertices at the same level, which is a contradiction.
In the case Figure \ref{fig15}, if the vertex $v$ is valence-four, $T$ must be a torus, as explained above.
Even if the vertex $v$ is valence-six, the same argument implies $T$ is a torus, which is a contradiction.
\end{proof}

\begin{figure}[ht]
\begin{center}
\caption{}
\vspace{.2cm}
\includegraphics[width=7.5cm]{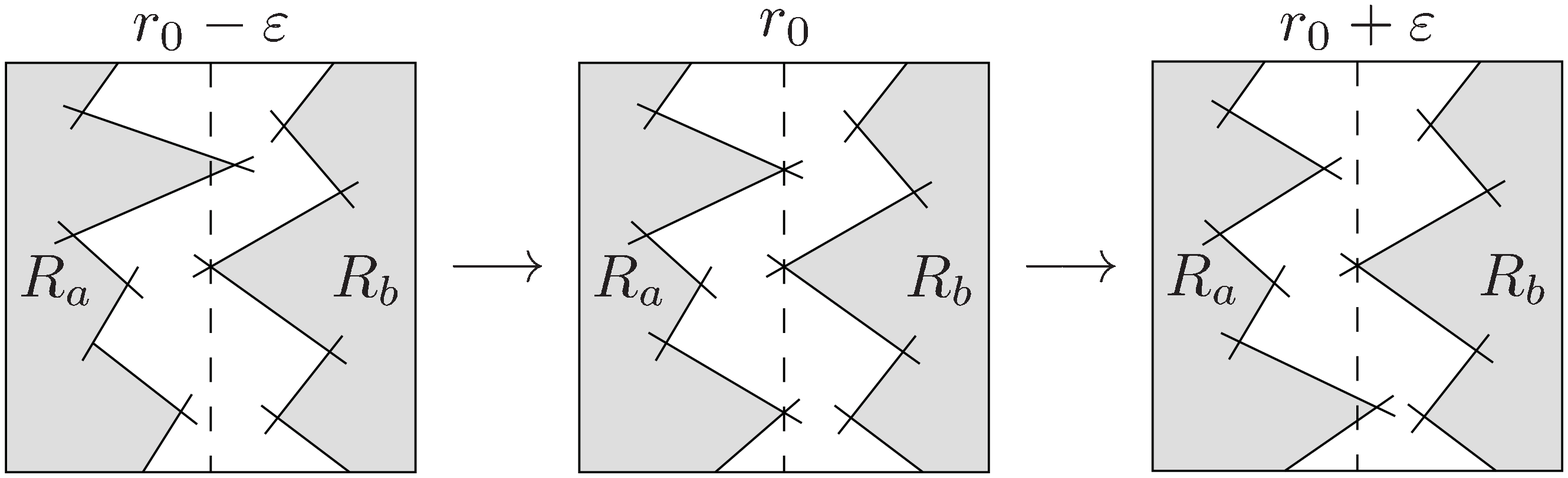}
\label{fig13}
\\
\caption{}
\vspace{.2cm}
\includegraphics[width=7.5cm]{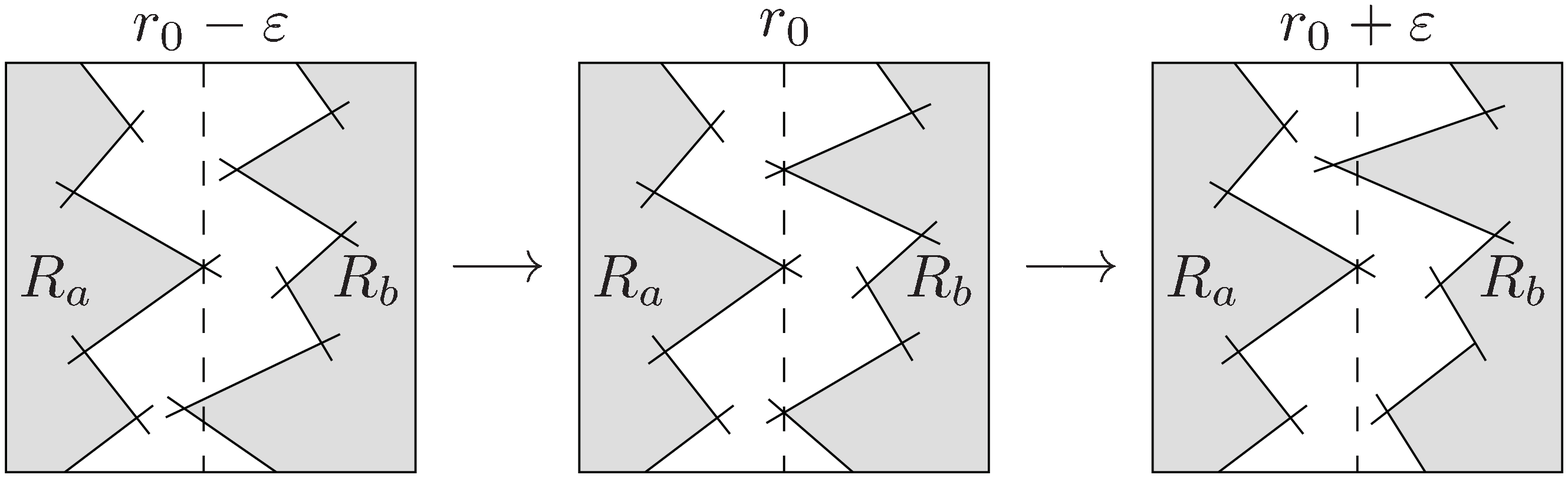}
\label{fig14}
\\
\caption{}
\vspace{.2cm}
\includegraphics[width=7.5cm]{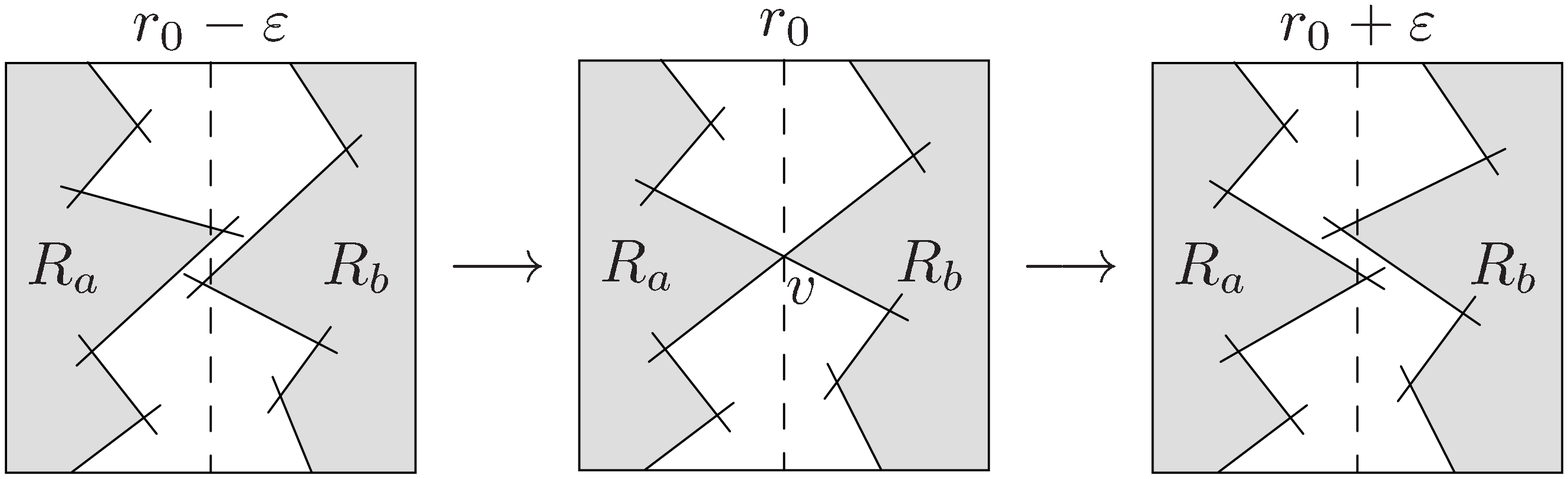}
\label{fig15}
\end{center}
\end{figure}

\section{Planar surfaces in a product space}

This section is for the final phase of the proof of the main theorem.
It may possibly be easy for the reader to take this section after a view of Section \ref{Main}.

Suppose $\Sigma $ is a closed, orientable, connected surface of genus $g$.
Let $W$ be the product space $\Sigma \times [s_-,s_+]$ where $s_-<s_+$.
Suppose $P$ is a separating, planar surface with $m_0$ components properly embedded in $W$.
Suppose $P$ separates $W$ into $W_-$ and $W_+$.
For each level $s\in [s_-,s_+]$, let $\Sigma ^{\pm }(s)$ be the intersection of $\Sigma \times \{ s\} $ with $W_{\pm }$.
We will focus on $\Sigma ^-(s_-)$ and $\Sigma ^+(s_+)$.
Let $g_-$and $g_+$ be the sum of the genera of all components of $\Sigma ^-(s_-)$ and $\Sigma ^+(s_+)$, respectively.

\begin{Lem}\label{planar}
$g\geq g_-+g_+$
\end{Lem}

\begin{proof}
We can assume $P$ is incompressible in $W$ because compressions of $P$ does not change $g_-$ or $g_+$.

Consider a component of $P$ which has all its boundary components on $\Sigma \times \{ s_-\} $.
Such a surface is $\partial $-parallel, i.e. it can be isotoped onto $\Sigma \times \{ s_-\} $ \cite[Corollary 3.2]{waldhausen}.
Whichever it is parallel to a component of $\Sigma ^-(s_-)$ or $\Sigma ^+(s_-)$, the component has no genus because $P$ is planar.
Therefore deleting the component of $W_-$ or $W_+$ between these parallel surfaces does not reduce $g_-$ or $g_+$.
Thus, it is sufficient to prove the lemma assuming all such component has been deleted.
In other words, we can assume every components of $P$ has the boundaries both on $\Sigma \times \{ s_-\} $ and $\Sigma \times \{ s_+\} $.

Let $m_{\pm }$ be the number of components of $\Sigma ^{\pm }(s_{\pm })$ and let $p_{\pm }$ be the number of boundary components of $\Sigma ^{\pm }(s_{\pm })$.
Then the Euler numbers of the surfaces concerned can be written as fallows:
\begin{align*}
&\chi (\Sigma )=2-2 g,\\
&\chi (\Sigma ^-(s_-))=2 m_--2 g_--p_-,\\
&\chi (\Sigma ^+(s_+))=2 m_+-2 g_+-p_+,\\
&\chi (P)=2 m_0-p_--p_+.
\end{align*}

Let $f:W\rightarrow [s_-,s_+]$ be a projection.
We can assume $P$ is in general position with respect to $f$.
Moreover, we can assume $P$ has been isotoped so that there are no extrema because every component of $P$ has the boundaries both on $\Sigma \times \{ s_-\} $ and $\Sigma \times \{ s_+\} $.
Write $s_1=s_-,s_{n+1}=s_+$ and let $s_2<s_3<\dots <s_n$ be the regular values for $f|_P$ such that there is a single critical value for $f|_P$ between $s_i$ and $s_{i+1}$ for each $i=1,2,\ldots,n$.
Write $P_i=P\cap f^{-1}([s_i,s_{i+1}])$ for each $i=1,2,\ldots,n$.
Each $P_i$ is a collection of annuli except for one pair of pants component of some of types in Figure \ref{fig17}.

\begin{figure}[ht]
\begin{center}
\caption{}
\vspace{.2cm}
\includegraphics[height=2cm]{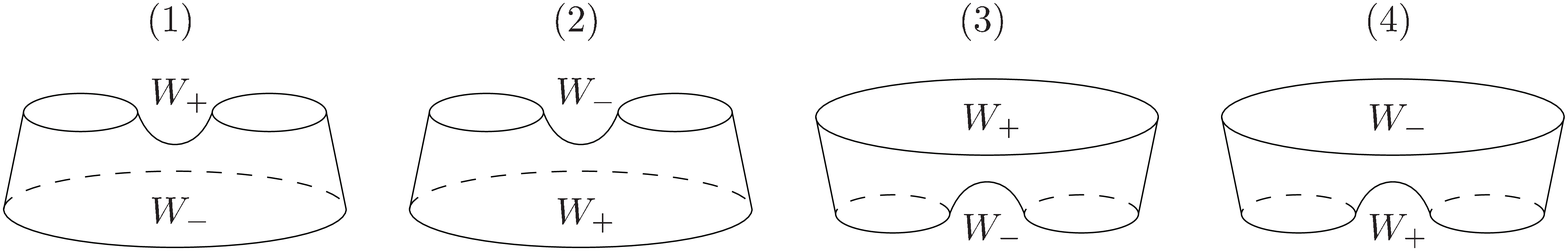}
\label{fig17}
\end{center}
\end{figure}

Consider the case where $P_i$ has a component of type $(1)$ for example.
The Euler number of $P_i$ is $-1$.
The surface $\Sigma ^+(s_{i+1})$ is homeomorphic to the union of $\Sigma ^+(s_i)$ and $P_i$.
Therefore the Euler number of $\Sigma ^+(s_{i+1})$ is one less than that of $\Sigma ^+(s_i)$.
Considering the other cases similarly, we obtain the following:
\begin{align*}
&\chi (P)=\sum ^n_{i=1}\chi (P_i)=-n_1-n_2-n_3-n_4,\\
&\chi (\Sigma ^+(s_+))-\chi (\Sigma ^+(s_-))=\sum ^n_{i=1}\{ \chi (\Sigma ^+(s_i))-\chi (\Sigma ^+(s_{i+1}))\} =-n_1+n_2-n_3+n_4
\end{align*}
where $n_j$ is the number of critical points of type $(j)$.

Because $\Sigma \times \{ s_-\} $ is the union of $\Sigma ^-(s_-)$ and $\Sigma ^+(s_-)$,
$$\chi (\Sigma )=\chi (\Sigma ^-(s_-))+\chi (\Sigma ^+(s_-)).$$
Applying above equations, we can arrive at a formula:
$$g=g_-+g_++1+m_0-m_--m_++n_2+n_4$$

Let $w_{\pm }$ be the number of components of $W_{\pm }$.
Then $w_-+w_+$ is the number of components of $W\setminus P$.
It implies
$$1+m_0\geq w_-+w_+.$$

Each of $m_-$ components of $\Sigma ^-(s_-)$ is contained in one of the $w_-$ components of $W_-$.
Let $W^0_-$ be a component of $W_-$ which contains $m^0_-$ components of $\Sigma ^-(s_-)$.
Observe the transformation of $W^0_-\cap \Sigma ^-(s)$ during the increasing of $s$ from $s_-$ to $s_+$.
Since $W^0_-$ is connected, there must be at least $m^0_--1$ critical points for $f|_{P\cap W^0_-}$ where two components of $W^0_-\cap \Sigma ^-(s)$ come to be connected.
Such critical points are type $(4)$.
Thus,
$$n_4\geq m_--w_-.$$
By the symmetrical argument,
$$n_2\geq m_+-w_+.$$

These inequalities immediately induce $g\geq g_-+g_+$.
\end{proof}

\section{The main theorem}\label{Main}

Johnson \cite{johnson2} constructed a counterexample for the Stabilization Conjecture by amalgamations of two Heegaard splittings with high distance along the torus boundaries.
We will make the same construction changing the place of torus boundaries by sphere boundaries.
By Proposition \ref{prop1}, an amalgamation along sphere boundaries is no other than a connected sum.
In this way, we arrive at the following conclusion.
Since Hempel \cite{hempel} showed that there exist Heegaard splittings with arbitrarily high distance, this immediately induces Theorem \ref{main}.

\begin{Thm}
Suppose $k\geq 2$ and $(T_i,G^-_i,G^+_i)$ is a Heegaard splitting of genus $k$ with distance at least $6k$ for a closed $3$-manifold $M_i$ for each $i=1,2$.
Let $(\Sigma _1,H^-_1,H^+_1)$ be the connected sum of $(T_1,G^-_1,G^+_1)$ and $(T_2,G^-_2,G^+_2)$.
Let\break $(\Sigma _2,H^-_2,H^+_2)$ be the connected sum of $(T_1,G^-_1,G^+_1)$ and $(T_2,G^+_2,G^-_2)$.
Then the stable genus of $(\Sigma _1,H^-_1,H^+_1)$ and $(\Sigma _2,H^-_2,H^+_2)$ is $3k$.
\end{Thm}

\begin{proof}
Since the genus of a connected sum is equal to the sum of the genera of original splittings, the genus of $(\Sigma _1,H^-_1,H^+_1)$ and $(\Sigma _2,H^-_2,H^+_2)$ is $2k$.
As remarked in \cite[Section 2]{hass}, the flip genus of any Heegaard splitting is at most twice the initial genus.
Therefore the Heegaard splitting $(T_2,G^-_2,G^+_2)$ become flippable after adding $k$ trivial handles.
It implies that adding $k$ trivial handles makes $(\Sigma _1,H^-_1,H^+_1)$ isotopic to $(\Sigma _2,H^-_2,H^+_2)$.
Thus, the stable genus is at most $3k$.
Then, we will show that the stable genus is at least $3k$.

Let $B_1$ and $ B_2$ be open balls in $G^+_1$ and $G^+_2$, respectively.
Write $M^*_i=M_i\setminus B_i$ and $G^{*+}_i=G^+_i\setminus B_i$ for each $i=1,2$.
The connected sum $M$ of $M_1$ and $M_2$ can be obtained by gluing $M^*_1$ and $M^*_2$ at their sphere boundaries.
$(T_i,G^-_i,G^{*+}_i)$ is a Heegaard splitting for a $3$-dimensional submanifold $M^*_i$ of $M$ for each $i=1,2$.
By Proposition \ref{prop1}, $(\Sigma _2,H^-_2,H^+_2)$ is the amalgamation of $(T_1,G^-_1,G^{*+}_1)$ and $(T_2,G^-_2,G^{*+}_2)$.
By Propositions \ref{prop1} and \ref{prop2}, $(\Sigma _1,H^-_1,H^+_1)$ is the amalgamation of $(T_1,G^-_1,G^{*+}_1)$ and the boundary stabilization of $(T_2,G^-_2,G^{*+}_2)$.

By Lemma \ref{1-9}, $(T_1,G^-_1,G^{*+}_1)$ spans itself positively.
By Lemma \ref{2-14}, $(\Sigma _2,H^-_2,H^+_2)$ spans $(T_1,G^-_1,G^{*+}_1)$ positively.
Similarly, $(\Sigma _2,H^-_2,H^+_2)$ spans $(T_2,G^-_2,G^{*+}_2)$ negatively and $(\Sigma _1,H^-_1,H^+_1)$ spans $(T_1,G^-_1,G^{*+}_1)$ positively.
By Lemmas \ref{1-9}, \ref{2-14} and \ref{2-16}, $(\Sigma _1,H^-_1,H^+_1)$ spans $(T_2,G^-_2,G^{*+}_2)$ positively.

Suppose $(\Sigma '_i,H^{\prime -}_i,H^{\prime +}_i)$ is a stabilization of $(\Sigma _i,H^-_i,H^+_i)$ for each $i=1,2$.
By Lemma \ref{2-12}, $(\Sigma '_i,H^{\prime -}_i,H^{\prime +}_i)$ spans $(T_1,G^-_1,G^{*+}_1)$ and $(T_2,G^-_2,G^{*+}_2)$ with the same signs as $(\Sigma _i,H^-_i,H^+_i)$.
If $(\Sigma '_1,H^{\prime -}_1,H^{\prime +}_1)$ and $(\Sigma '_2,H^{\prime -}_2,H^{\prime +}_2)$ are isotopic, the isotopy takes $H^{\prime -}_1$ to either $H^{\prime -}_2$ or $H^{\prime +}_2$.

Consider the case where the isotopy takes $H^{\prime -}_1$ to $H^{\prime -}_2$ and $H^{\prime +}_1$ to $H^{\prime +}_2$.
The Heegaard splitting $(\Sigma '_1,H^{\prime -}_1,H^{\prime +}_1)$ spans $(T_2,G^-_2,G^{*+}_2)$ positively and negatively.
If $(\Sigma '_1,H^{\prime -}_1,H^{\prime +}_1)$ splits $(T_2,G^-_2,G^{*+}_2)$, Lemma \ref{splitting} implies that the genus of\break $(\Sigma '_1,H^{\prime -}_1,H^{\prime +}_1)$ is at least $3k$.
By Lemma \ref{isotopy}, we can assume there is a pair of sweep-outs $f$ and $g_2$ representing $(\Sigma '_1,H^{\prime -}_1,H^{\prime +}_1)$ and $(T_2,G^-_2,G^{*+}_2)$ such that $f$ spans $g_2$ positively and negatively.
By Lemma \ref{spanning}, $f^{-1}(s_2) \cap M^*_2$ has genus at least $2k$ for some value $s_2\in (-1,1)$.
For a sweep-out $g_1$ representing $(T_1,G^-_1,G^{*+}_1)$, if $f$ splits $g_1$ then Lemma \ref{splitting} can be applied again.
Therefore we can assume $f$ spans $g_1$.
By Lemma \ref{spanning}, $f^{-1}(s_1) \cap M^*_1$ has genus at least $k$ for some value $s_1\in (-1,1)$.
Assume $s_1<s_2$ without loss of generality.
The intersection $M^*_1\cap M^*_2\cap f^{-1}([s_1,s_2])$ is a separating, planar surface properly embedded in a product space $f^{-1}([s_1,s_2])$.
By Lemma \ref{planar}, the genus of $\Sigma '_1$ is at least $k+2k=3k$.

On the other hand, when the isotopy takes $H^{\prime -}_1$ to $H^{\prime +}_2$ and $H^{\prime +}_1$ to $H^{\prime -}_2$, the Heegaard splitting $(\Sigma '_1,H^{\prime -}_1,H^{\prime +}_1)$ spans $(T_1,G^-_1,G^{*+}_1)$ positively and negatively.
The same argument implies that the genus of $\Sigma '_1$ is at least $3k$.
Thus, any common stabilization of $(\Sigma _1,H^-_1,H^+_1)$ and $(\Sigma _2,H^-_2,H^+_2)$ has genus at least $3k$.
\end{proof}

\end{document}